\RequirePackage{fix-cm}
\documentclass[smallextended,epsf,epsfx]{article} 
\newtheorem{theorem}{Theorem}

\newtheorem{example}{Example}
\newenvironment{proof}{{\bf Proof. }\begin{list}{}{}\item[]}{{\bf Q.E.D.}\end{list}}
\usepackage{graphicx}
\usepackage{listings}
\usepackage{mathptmx,amsmath,amssymb}      
\usepackage{latexsym}
\begin{document}

\title{Leap Gradient Algorithm  \thanks{ Arizona State University - School of Mathematical and Statistical Sciences}
}


\author{Sergey Nikitin       
}




\maketitle

\begin{abstract}
The paper proposes a new algorithm for solving global univariate optimization problems. The algorithm does not require convexity of the target function. For a broad variety of target functions after performing (if necessary) several evolutionary leaps the algorithm naturally becomes the standard descent (or ascent) procedure near the global extremum. Moreover, it leads us to an efficient numerical method for calculating the global extrema of univariate real analytic functions. 

\end{abstract}

\section{Introduction}
\label{intro}

The problem of finding global extrema (maxima and minima) for a univariate real function is important for variety of real world applications. For example, it arises in electric engineering \cite{Hamacher},\cite{Johnson}, \cite{Lam} in computer science \cite{Calvin}, \cite{Calvin_Zilinkas}, \cite{Kalra} and in various other fields (see \cite{Sergeyev} for further references) . In many industrial applications the global optimization algorithm is expected to operate in real time while simultaneously, finding the global extremum of non-convex functions exhibiting large number of local sub-extrema. 

The problem of efficiently finding a function's global extremum has been historically challenging. One of the first solutions, Zero Derivative Method (ZDM), was proposed by  Pierre de Fermat (1601-1665). His main idea was to look for the global extremum among critical points: the points where the derivative of the target function is zero. Despite its theoretical significance, Fermat's proposed method (ZDM) is limited by the numerical difficulties imposed by finding critical points.

 One of the leading global optimization approaches adopted by many industrial applications is a brute-force search or exhaustive search for the global extremum (Brute-Force Search (BFS)). It is simple to implement but its performance linearly depends on the complexity of the target function and the size of the search area.

   A plethora of optimization methods have been developed for various types of target functions. Among them Piyavskii-Shubert Method (PSM) occupies a special place \cite{Piyavskii_67}, \cite{Piyavskii_72}, \cite{Shubert}. It is one of the few procedures that delivers the global extremum for a univariate function and at the same time it exhibits reasonable performance as long as the respective Lipschitz constant is of a modest value. On the other hand, the method is very sensitive to the size of the Lipschitz constant: its performance sharply diminishes for large Lipschitz constants. For this reason, accelerations and improvements of PSM were developed in the following papers \cite{Ellaia}, \cite{Kvasov}, \cite{Sergeyev}, \cite{YSergeyev} . Numerical experiments presented in this paper show that Leap Gradient Algorithm (referred as LGA) significantly outperforms  PSM together with its modifications and improvements (from \cite{Ellaia},  \cite{Sergeyev}, \cite{YSergeyev}) when finding global extrema of polynomials.          
 
   The method of gradient descent (see, e.g. \cite{Guler}, \cite{KantorovichAkilov}, \cite{Nesterov}) is widely used to solve various practical optimization problems. The main advantage of the gradient descent algorithm is its simplicity and applicability to a wide range of practical problems. On the other hand, gradient descent has limitations imposed by the initial guess of a starting point and then its subsequent conversion to a suboptimal solution. This paper gives practical recipes on how to overcome those limitations for a univariate function and how to equip the gradient descent algorithm with abilities to converge to a global extremum. It is achieved via evolutionary leaps towards the global extremum. LGA neither requires the knowledge of the Lipschitz constant nor convexity conditions that are often imposed on the target function. Moreover, LGA naturally becomes the standard gradient descent procedure when the target function is convex or when the algorithm operates in the close proximity to the global extremum.

The recursive application of LGA yields an efficient algorithm for calculating global extrema for univariate polynomials. LGA does not intend to locate any critical points (like ZDM) instead it follows gradient descent (ascent) until a local extremum is reached and then performs an evolutionary leap towards the next extremum. As far as performance is concerned, numerical experiments conducted for univariate polynomials show that LGA outperforms BFS, ZDM and PSM with all its modifications from \cite{Ellaia},  \cite{Sergeyev}, \cite{YSergeyev}.

The layout of this publication is as follows.

Section \ref{intro} is the introduction.

Section \ref{sec:2} introduces LGA for univariate functions.

Section \ref{LGArecursion} explores the recursive application of LGA.

Section \ref{MultiW} describes a recursive implementation of LGA for polynomials.  

Section \ref{RealAnalytic} presets an implementation of LGA for real analytic univariate functions.


Section \ref{numerical_experiments} reports on numerical experiments with LGA. It compares LGA with ZDM, BFS, PSM and accelerations of PSM presented in \cite{Ellaia},  \cite{Sergeyev}, \cite{YSergeyev}. Polynomial roots for ZDM are calculated with the help of Laguerre's Method \cite{Adams}, \cite{Armijo}.


Section \ref{Acknowledgment} expresses gratitude to professionals who spent their time contributing to the paper, reviewing various parts of the publication and making many valuable suggestions and improvements.

Section \ref{appendix} finalizes the publication with a snapshot of the working and thoroughly tested source code for LGA.

\section{LGA: leaps towards the global extremum}
\label{sec:2}
Let $f(x)$ be known and well defined real function on $[a,\;b],$ a closed interval of real numbers. Consider the optimization problem
$$
f(x) \quad \to \quad  \min_{x\in [a,\;b]} 
$$
LGA can solve it with a given precision $h>0$ if its solution exists.

\vspace{0.1cm}

\paragraph{ \bf Leap Gradient Algorithm (LGA).}

\begin{itemize}
\item[\bf STEP 0.] Set
$$
x_0 =a.
$$
\item[\bf STEP 1.] Iterate
$$
x_{k+1} = x_k + h,
$$ 
as long as 
$$
f(x_k + h)\le f(x_k) 
$$
and
$$
x_k < b.
$$
If $x_k \ge b$ then {\bf STOP} and take $(b, f(b))$ as an estimate of the argument and the value for the global minimum.

If 
\begin{equation}
\label{gradient_decent}
f(x_k + h) > f(x_k) 
\end{equation}
 then proceed with {\bf STEP 2}.
\item[\bf STEP 2.] 
Let $x_k^{\star}$ be the solution of the following optimization problem
\begin{equation}
\label{optimization_proc}
\frac{f(x) - f(x_k)}{x - x_k}  \quad \to \quad \min_{x \in [x_k,\;b]}
\end{equation}
If 
$$
x_{k}^{\star} \ge x_k \;\;\;\mbox{and}\;\;\;f(x_k^{\star}) \ge f(x_k)
$$
then {\bf STOP} and $(x_k, f(x_k))$ is an estimate of the argument and the value for the global minimum. 

If 
$$
x_{k}^{\star} > x_k \;\;\;\mbox{and}\;\;\;f(x_k^{\star}) < f(x_k)
$$
 then $x_{k+1} = x_k^{\star}$ (LGA performs {\it  an evolutionary leap}) and go to {\bf STEP 1}.
\end{itemize}

LGA is illustrated in Fig. ~\ref{fig:1}.
\input epsf
\setlength{\unitlength}{1cm}
\begin{figure}
\epsfxsize=11cm
\epsfysize=6cm
\epsfbox{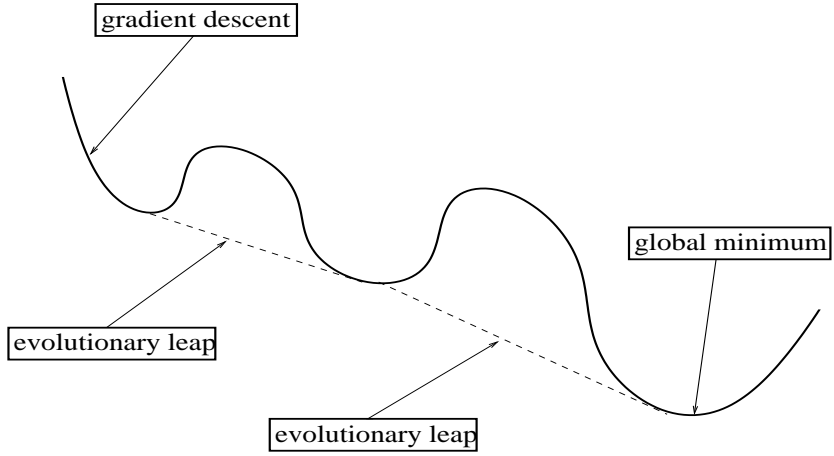}
\caption{Leap Gradient Algorithm (LGA) for a univariate function}
\label{fig:1}       
\end{figure}
If the target function is twice continuously differentiable on $(a,\;b)$ then the number of inflection points from $(a,\;b)$ is related to the number of evolutionary leaps performed by LGA (see {\bf STEP 2}). 
\begin{theorem}
\label{inflection_and_leaps}
Let $f(x)$ be twice continuously differentiable on $(a,\;b)$. Let $x^{\star}_k$ be the value of the evolutionary leap ({\bf STEP 2} of LGA). 

If 
$$
a<x_k<x^{\star}_k<b
$$
then $(x_{k-1},\;x^{\star}_k]$ contains at least two points where $(\frac{d}{dx})^2f(x)$ is equal to zero. 
\end{theorem}
\begin{proof}
Evolutionary leaps from $x_k$ (on $(a,\;b))$ are solutions of the following equation.
\begin{equation}
\label{leap_formula}
f(x) - f(x_k) = \frac{d}{dx}f(x) \cdot  (x - x_k) 
\end{equation}
where $x \in \; [x_k,\;b).$ 

Indeed, if $f(x)$ is continuously differentiable then the necessary condition for $x$ to be the solution for (\ref{optimization_proc}) on $(x_k,\;b)$ is
$$
\frac{d}{dx}\big( \frac{f(x) - f(x_k)}{x-x_k}\big) =0.
$$

Differentiating yields
$$
\frac{\frac{d}{dx}f(x)}{x-x_k} -\frac{f(x)-f(x_k)}{(x-x_k)^2}=0.
$$
After multiplying the equation with $(x-x_k)^2$ we obtain (\ref{leap_formula}).

Since $x^{\star}_k$ is the solution of the problem (\ref{optimization_proc}) and $f(x)$ is twice differentiable we conclude that
$$
\big(\frac{d}{dx}\big)^2 (\frac{f(x) - f(x_k)}{x - x_k})\left. \right|_{x=x^{\star}_k} \ge 0 
$$ 
Differentiating yields
$$
\frac{\big(\frac{d}{dx}\big)^2 f(x)}{(x-x_k)} -\frac{2}{(x-x_k)^3}\cdot (\big(\frac{d}{dx} f(x)\big) \cdot (x-x_k) - (f(x)-f(x_k))  ) \ge 0
$$
when $x=x^{\star}_k.$ Hence, making use of (\ref{leap_formula}) we obtain 
\begin{equation}
\label{rightplus}
\left.\big(\frac{d}{dx}\big)^2 f(x)\right|_{x=x^{\star}_k} \ge 0.
\end{equation}

According to {\bf STEP 2} of LGA for the evolutionary leap $x^{\star}_k$ we have
$$
f(x^{\star}_k)<f(x_k).
$$
That together with (\ref{leap_formula}) implies 
\begin{equation}
\label{rt}
\left.\big(\frac{d}{dx}\big) f(x)\right|_{x=x^{\star}_k} <0.
\end{equation}
On the other hand, (\ref{gradient_decent}) yields the existence of $\xi$ such that 
$$
\left.\big(\frac{d}{dx}\big) f(x)\right|_{x=\xi} >0 \;\;\mbox{ and } \xi \in (x_k,\;x_k+h)
$$ 
That together with (\ref{rt}) yield the existence of $\tilde { x} \in ( x_{k}, x^{\star}_k)$ where
\begin{equation}
\label{midminus}
\left.\big(\frac{d}{dx}\big)^2 f(x)\right|_{x=\tilde{ x}} < 0.
\end{equation}
Taking into account that $x_k>a$ is obtained after {\bf STEP 1} of LGA and
$$
f(x_k)- f(x_{k-1}) \le 0, \;\;\; \;\;\;
f(x_k+h) -  f(x_k)  > 0 
$$
we obtain the existence of $\bar{\bar{x}} \in (x_{k-1},\:x_k+h)$ where
\begin{equation}
\label{leftplus}
\left.\big(\frac{d}{dx}\big)^2 f(x)\right|_{x=\bar{\bar{ x}}} > 0.
\end{equation}
Under the conditions of the theorem 
$$
\big(\frac{d}{dx}\big)^2 f(x)
$$
is continuous on $(a,\;b)$ and the statement follows from (\ref{rightplus}), (\ref{midminus}) and (\ref{leftplus}). 

\end{proof}

\vspace{0.1cm}

That means (in a generic situation) a non trivial evolutionary leap inside $(x_k,\;b)$ is only possible over two inflection points of a target function. 

\begin{example}
Consider the optimization problem
\begin{equation}
\label{4thdegree}
f(x)=x^4+a \cdot x^3 + b\cdot x^2 +c\cdot x+d \quad \to \quad \min_{x \in \mathbb{R}} 
\end{equation}
where $a,\;b,\;c,\;$ and $d$ are real numbers. 
If
$$
3a^2 > 8b
$$
then  
$$
\frac{d^2}{dx^2} f(x) = 4\cdot 3 \cdot x^2 + 3\cdot 2 \cdot a\cdot x + 2\cdot b
$$
has two different zeroes. Hence, by Theorem \ref{inflection_and_leaps}, LGA might need to perform a single evolutionary leap (over two inflection points) in order to  solve the optimization problem.

Otherwise, $3a^2 \le 8b$, LGA coincides with the standard gradient decent.
\end{example}

\section{Recursive leap gradient procedure}
\label{LGArecursion}
LGA replaces the optimization problem
$$
f(x) \;\;\to\;\;\min_{x \in [a,\;b]}
$$
with
$$
\frac{f(x) - f(x_k)}{x - x_k} \;\;\to \;\; \min_{x \in [x_k.\;b]}
$$
where $x_k$ is calculated at the previous step of LGA. It leads us to the following recursive procedure executed at each iteration of LGA.
\begin{equation}
\label{iter_1}
g_0(x) = f(x)
\end{equation}
and
\begin{equation}
\label{iter_2}
g_{m}(x) = \frac{g_{m-1}(x) - g_{m-1}(x_k)}{x - x_k} .
\end{equation}
In order to complete an iteration of LGA for $g_{m-1}(x)$ one needs to calculate
$$
g_{m} (x) \;\;\to\;\; \min_{x \in [x_k.\;b]}
$$
and LGA finds the minimum of $g_{m}(x).$ If finding the minimum for $g_m(x)$ is obvious, then the iteration of LGA is completed after $m$ recursive steps.  

The next theorem shows when an LGA iteration is completed after a finite number of recursive steps.

\begin{theorem}
If a real function $f(x)$ is at least $n$-times continuously differentiable on the segment $[a,\;b]$ and
$$
\frac{d^n}{dx^n}f(x) \ge 0 \;\;\;\forall \;\;x \in \;[a,\;b]
$$
then each iteration of LGA for
$$
f(x) \;\;\to\;\;\min_{x \in [a,\;b]}
$$ 
is completed after not more than $n-1$  recursive steps.
\end{theorem}

\begin{proof}

\vspace{0.1cm}

The function can be represented by Taylor expansion centered at $x_k,$
$$
f(x) = f(x_k) + \sum_{j=1}^{n-1} \frac{1}{j!} \cdot \frac{d^j}{dx^j}f(x_k) \cdot (x-x_k)^j + \int_0^1 \frac{d^n}{dx^n}f(x_k+t\cdot (x-x_k))\cdot \frac{(x-x_k)^n(1-t)^{n-1}}{(n-1)!}dt
$$
for all $x \in \;[x_k,\;b].$
In accordance with notations (\ref{iter_1}) and (\ref{iter_2}) we have
\begin{eqnarray*}
g_m(x) &=& \frac{1}{m!} \cdot \frac{d^m}{dx^m}f(x_k) + \sum_{j=m+1}^{n-1} \frac{1}{j!} \cdot \frac{d^j}{dx^j}f(x_k) \cdot (x-x_k)^{j-m}+\\
&&  \int_0^1 \frac{d^n}{dx^n}f(x_k+t\cdot (x-x_k))\cdot \frac{(x-x_k)^{n-m}(1-t)^{n-1}}{(n-1)!}dt 
\end{eqnarray*}
and
$$
g_{n-1}(x) = \frac{1}{(n-1)!}\frac{d^{n-1}}{dx^{n-1}}f(x_k) + \int_0^1 \frac{d^n}{dx^n}f(x_k+t\cdot (x-x_k))\cdot \frac{(1-t)^{n-1}}{(n-1)!}dt \cdot (x - x_k).
$$
Under the conditions of the theorem $g_{n-1}(x)$ achieves its minimum on $[x_k,\;b]$ at $x_k.$

\end{proof}

Recursive LGA is an efficient numerical method for finding global extrema of univariate polynomials.

\begin{theorem}
For any polynomial
$$
p(x) = p_0+p_1 \cdot x + p_2 \cdot x^2 +\dots + p_n \cdot x^n
$$
on a segment $[a,\;b]$ recursive LGA delivers the global extremum for $p(x)$ in a finite number of steps. 
\end{theorem}

\begin{proof}

\vspace{0.1cm}

The proof is conducted by mathematical induction with respect to the degree of a polynomial. As a basis of mathematical induction and for the purpose of illustrations let us consider finding the global minimum on $[a,\;b]$ for the quadratic polynomial
$$
p(x)= p_0 + p_1  \cdot x  + p_2 \cdot x^2 \;\;\;(p_2 \not= 0)
$$
In accordance with notations (\ref{iter_1}), (\ref{iter_2}) 
\begin{eqnarray*}
g_0(x) &=& p_0 + p_1  \cdot x  + p_2 \cdot x^2\\
g_1(x) &=&  p_1 + 2p_2 \cdot a + p_2 \cdot (x-a).
\end{eqnarray*}
If $p_2 >0$ then $g_1(x)$ has its minimum at $x_0=a.$  If 
$$
p_1 + 2p_2 \cdot a \ge 0
$$
then according to LGA the global minimum is reached ad $x=a$ and the procedure stops. Otherwise,
$$
p_1 + 2p_2 \cdot a < 0
$$
LGA leads us to $x_1 = a + h,$ where the step size $h$ is dictated by the required precision of LGA. We follow the standard gradient descent when calculating $x_1,\;x_2,\;x_3,\;\dots x_k$ as long as 
$$
p_1 + 2p_2 \cdot x_k < 0.
$$
The gradient descent either stops at $b$ and the global minimum is located at $b$ or, according to LGA, it stops at the first $x_k$ such that
$$
p_1 + 2p_2 \cdot x_k \ge 0
$$
and $x_k$ delivers the global minimum.

If $p_2 <0$ then the minimum for $g_1(x)$ is located at $b.$ According to LGA, if 
$$
p_1 + 2p_2 \cdot a + p_2 \cdot (b-a) \ge 0
$$
then the minimum is at $a.$ Otherwise, the minimum is at $b.$ The basis of the mathematical induction is established.

The step of mathematical induction follows directly from recursive LGA. Indeed, suppose that recursive LGA delivers, in a finite number of steps, the global minimum for any polynomial of degree less than $n.$ However, following notations (\ref{iter_1}) and (\ref{iter_2}), in order to calculate the global minimum for a polynomial $g_0(x)$ of $n$-th degree we need to calculate the global minimum for $g_1(x),$ a polynomial of $(n-1)$-st degree. Hence, the statement follows by mathematical induction.

\end{proof}

\section{LGA: numerical recursive procedure for polynomial extrema.}
\label{MultiW}

Horner's algorithm  (see, e.g., \cite{Higham}, \cite{Knuth}) plays the central role in implementation of recursive LGA for polynomials. 
LGA reduces the optimization problem
$$
f(x) \;\;\to\;\;\min_{x \in [a,\;b]}
$$
to
$$
\frac{f(x) - f(x_k)}{x - x_k} \;\;\to \;\; \min_{x \in [x_k,\;b]}
$$
where $x_k$ is calculated at the previous step of LGA. If 
$$
f(x) = p_0+p_1 \cdot x + p_2 \cdot x^2 + \dots + p_n \cdot x^n
$$
is a polynomial with real coefficients then so is
$$
\frac{f(x) - f(x_k)}{x - x_k} = q_0+q_1 \cdot x + q_2 \cdot x^2 + \dots + q_{n-1} \cdot x^{n-1}
$$
where coefficients $q_0,\;q_1, \dots q_{n-1}$ are calculated  as follows. 

\vspace{0.1cm}

\paragraph{\bf Horner's Algorithm }
\begin{itemize}
\item[$\bullet$]$q_{n-1} = p_n$
\item[$\bullet$] $q_{j-1} = x_k \cdot q_j + p_j,\;\; \mbox{ where } \;\;j=n-1,\;n-2,\;\dots, \; 1 $
\end{itemize}

Recursive LGA (\ref{iter_1}),\; (\ref{iter_2})  described in section \ref{LGArecursion} is reduced to a finite number of iterations for Horner's Algorithm until the resulted polynomial is either a linear or a quadratic function. Then the solution of the optimization problem is trivial and therefore the recursive procedure delivers the global extremum.

Performing an evolutionary leap is a numerically expensive operation. Theorem \ref{inflection_and_leaps} shows that each LGA leap inside the interval is a jump over two zeroes of the second derivative of the target function. That allows to improve the performance of LGA for polynomials by limiting the number of evolutionary jumps by at most $n-2,$ where $n$ is the degree of the target polynomial. The respective modification of LGA for
$$
p_0+p_1 \cdot x + p_2 \cdot x^2 + \dots + p_n \cdot x^n \;\;\to \;\; \min_{x \in [a,\;b]}
$$
 is as follows. 

\paragraph{\bf LGA for polynomials}

\begin{itemize}
\item[\bf STEP 0.] If the degree of the polynomial is $1,$ then return
$$
(a,\;p_0+p_1 \cdot a) \;\;\mbox{ if }\;\;\;p_0+p_1 \cdot a \le \;p_0+p_1 \cdot b
$$
Otherwise, return 
$$
(b,\;p_0+p_1 \cdot b)
$$
as the argument, value pair for the global minimum.

If the degree is equal to $2,$
$$
P(x)=p_0+p_1 \cdot x + p_2 \cdot x^2.
$$
If $p_2 > 0$ then return $(b,\;P(b))$ for 
$$
-\frac{p_1}{2\cdot p_2}\ge b,
$$
and  return $(a,\;P(a))$ when
$$
-\frac{p_1}{2\cdot p_2}\le b.
$$
Otherwise, return 
$$
(c,\;P(c)), 
$$
where $c = -\frac{p_1}{2\cdot p_2}.$

 If the degree of the polynomial is larger than $2,$ then set 
$$
\mbox{\bf Number\_of\_jumps} = 0,
$$
$$
x_0 =a.
$$
\item[\bf STEP 1.] Iterate
$$
x_{k+1} = x_k + h,
$$ 
as long as 
$$
q_{k0}+q_{k1} \cdot (x_k+h) + q_{k2} \cdot (x_k+h)^2 + \dots + q_{k n-1} \cdot (x_{k}+h)^{(n-1)} \le 0 
$$
and
$$
x_k < b,
$$
where $(q_{k0},\;q_{k1},\;\dots\; q_{k n-1})$ are calculated with Horner's algorithm.
\begin{itemize}
\item[$\bullet$]$ q_{k n-1}=p_n$
\item[$\bullet$]$ q_{k i-1}= x_k\cdot q_{k i} + p_i\;\;\mbox{ for }\;\;i=1,\;2,\;\dots n-1$
\end{itemize}

If $x_k \ge b$ then {\bf STOP} and return $b$ as the argument and 
$$
p_0+p_1 \cdot b + p_2 \cdot b^2 + \dots + p_n \cdot b^n
$$
as the value for the estimate of the global minimum.

If 
$$
q_{k0}+q_{k1} \cdot (x_k+h) + q_{k2} \cdot (x_k+h)^2 + \dots + q_{k n-1} \cdot (x_{k}+h)^{(n-1)} > 0 
$$
 then proceed with {\bf STEP 2}.
\item[\bf STEP 2.] 
If LGA already performed $n-2$ evolutionary jumps, $\mbox{\bf Number\_of\_jumps} \ge n-2,$ then {\bf STOP} and return $x_k$ and 
$$
p_0+p_1 \cdot x_k + p_2 \cdot x_k^2 + \dots + p_n \cdot x_k^n
$$
as the argument, value pair for the estimate of the global minimum. 

Otherwise, recursively apply LGA to
\begin{equation}
\label{recursive_LGA}
q_{k0}+q_{k1} \cdot x + q_{k2} \cdot x^2 + \dots + q_{k n-1} \cdot x^{(n-1)} \;\;\to \;\;  \min_{x \in [x_k,\;b]}
\end{equation}

Let $x_k^{\star}$ be the solution of (\ref{recursive_LGA}). 
If 
$$
x_{k}^{\star} \ge x_k 
$$
and
$$
q_{k0}+q_{k1} \cdot (x_k^{\star}) + q_{k2} \cdot (x_k^{\star})^2 + \dots + q_{k n-1} \cdot (x_{k}^{\star})^{(n-1)} \ge 0 
$$
then {\bf STOP} and return $x_k,$
$$
p_0+p_1 \cdot x_k + p_2 \cdot x_k^2 + \dots + p_n \cdot x_k^n
$$
as an argument, value estimate of the global minimum. 

If 
$$
x_{k}^{\star} > x_k \;\;\;
$$
and
$$
q_{k0}+q_{k1} \cdot (x_k^{\star}) + q_{k2} \cdot (x_k^{\star})^2 + \dots + q_{k n-1} \cdot (x_{k}^{\star})^{(n-1)} < 0 
$$
then, by Theorem \ref{inflection_and_leaps}, increment $\mbox{\bf Number\_of\_jumps}$ by one if $x_k$ equals to $a$ otherwise by two. After that LGA performs {\it  an evolutionary leap} by setting
$$
x_{k+1} = x_k^{\star}.
$$
and proceeding with {\bf STEP 1}.
\end{itemize}

A polynomial
$$
p_0+p_1 \cdot x_k + p_2 \cdot x_k^2 + \dots + p_n \cdot x_k^n
$$
in LGA is evaluated with Horner's algorithm as follows.

Set 
$$
v= p_n.
$$
Then iterate
$$
v= v\cdot x + p_{j}\;\;\;\mbox{ for }\;\;\;j=n-1,\;n-2,\;\dots\;1,\;0
$$
and $v$ is the value of the polynomial at $x.$ 

Interested reader will find a snapshot of the working and thoroughly tested source code of LGA in Appendix of this paper.

\section{Global minimum of a univariate real analytic function}
\label{RealAnalytic}
Consider the optimization problem
\begin{equation}
\label{optimization_problem_analytic}
f(x) \;\;\to\;\;\min_{x \in [-1,\;1]}
\end{equation}
where $f(x)$ is a univariate real analytic function on $[-1,\;1].$ That means
\begin{equation}
\label{Taylor_expansion}
f(x) = \sum_{j=0}^\infty \frac{x^j}{j!}(\frac{d}{dx})^jf(0)
\end{equation}
and the series uniformly converges to $f(x)$ on $[-1,\;1].$
$$
(\frac{d}{dx})^jf(0)
$$
denotes the value of the derivative
$$
(\frac{d}{dx})^jf(x)
$$
at $x=0.$

This section proposes a numerical procedure for solving (\ref{optimization_problem_analytic}). The procedure is based on LGA. Namely, (\ref{optimization_problem_analytic})  is replaced with
\begin{equation}
\label{optimization_problem_LGA}
P_f^n(x) \;\;\to\;\;\min_{x \in [-1,\;1]}
\end{equation} 
where
$$
P_f^n(x) = \sum_{j=0}^n \frac{x^j}{j!} (D_h)^jf(0)
$$
and 
\begin{eqnarray*}
D_h^0f(x) &=& f(x),\\
D_hf(x) &=& \frac{f(x+h) - f(x-h)}{2\cdot h},\\
D_h^jf(x) &=& D_h(D_h^{j-1}f(x)) \mbox{ for } \; j=0,\;1,\;\dots
\end{eqnarray*}
The numerical algorithm is based on the following statement.

\begin{theorem}
\label{real_analytic_LGA}
For any $x \in [-1,\;1]$
\begin{equation*}
f(x) - P_f^n(x) = 
 \int_0^1 (\frac{d}{dx})^{n+1} f(t\cdot x)\cdot x^{n+1}\frac{(1-t)^{n}} {n!} dt - \sum_{j=1}^n x^j\cdot \int_0^1 \frac{(1-t)^{j+1}}{j!\cdot (j+1)!}\cdot (\frac{d}{dt})^{j+2}(t^j\cdot D_{t\cdot h}^jf(0))dt. 
\end{equation*}
\end{theorem}
\begin{proof}
It follows from Taylor expansion with the integral remainder term that
$$
f(x) - P_f^n(x) = \sum_{j=0}^n \frac{x^j}{j!} ((\frac{d}{dx})^jf(0) - (D_h)^jf(0)) + \int_0^1 (\frac{d}{dx})^{n+1} f(t\cdot x)\cdot x^{n+1}\frac{(1-t)^{n}} {n!} dt 
$$
In order to calculate 
\begin{equation}
\label{difference}
(\frac{d}{dx})^jf(0) - (D_h)^jf(0)
\end{equation}
let us justify the following statements with the help of mathematical induction.
\begin{eqnarray}
\label{first}
(D_h)^jx^k &=& 0\;\;\mbox{ for }\;\;0\le k < j\\
\label{second}
(D_h)^jx^j &=& j!\\
\label{third}
(D_h)^jx^{j+1} &=& 0\;\;\mbox{ for }\;\;x=0.
\end{eqnarray}
The basis of mathematical induction follows from
\begin{eqnarray*}
(D_h)^jx^k &=& 0\;\;\mbox{ for }\;\; k = 0 \;\;\mbox{ and } \;\;j>0\\
(D_h)x &=& \frac{(x+h)-(x-h)}{2\cdot h}=1\\
(D_h)x^{2} &=& \frac{(x+h)^2 - (x-h)^2}{2\cdot h} = 0\;\;\mbox{ for }\;\;x=0.
\end{eqnarray*}
The step of mathematical induction for each of the statements (\ref{first}), (\ref{second}), (\ref{third}) is as follows.

      Suppose that (\ref{first}) is true for $j\le m$ and $k<j.$ To prove that it remains true for $j=m+1$ and $k<m+1$ consider
$$
(D_h)^{m+1}x^k = (D_h)^m(\frac{(x+h)^k-(x-h)^k}{2\cdot h})= (D_h)^m(\frac{1}{2h}\sum_{s=1}^k C_s^k(1 - (-1)^s)x^{k-s}\cdot h^s)
$$
and
$$
(D_h)^m(\frac{1}{2h}\sum_{s=1}^k C_s^k(1 - (-1)^s)x^{k-s}\cdot h^s) = \sum_{s=1}^k C_s^k(1 - (-1)^s)(D_h)^m(x^{k-s}) \cdot h^s = 0
$$
due to the assumption of the mathematical induction $(D_h)^m(x^{k-s})=0$ for $s=1,\dots,k.$ The statement (\ref{first}) follows.

       Assume that (\ref{second}) is true for $j=m.$ Consider 
$$
(D_h)^{m+1} x^{m+1} = (D_h)^m(\frac{(x+h)^{m+1}-(x-h)^{m+1}}{2\cdot h})
$$
and (\ref{first}) together with the assumption of the mathematical induction yield
$$
(D_h)^{m+1} x^{m+1}=\sum_{s=1}^{m+1}C_s^{m+1}(1 - (-1)^s)(D_h)^m(x^{m+1-s}) \cdot \frac{h^{s-1}}{2} = (m+1)!.
$$ 
The statement (\ref{second}) is established.

        Assume that (\ref{third}) is valid for $j=m.$ Then
$$
(D_h)^{m+1} x^{m+2} = (D_h)^m(\frac{(x+h)^{m+2}-(x-h)^{m+2}}{2\cdot h})
$$
By the assumption of the mathematical induction and taking into account (\ref{first}), (\ref{second}) we have 
$$
(D_h)^{m+1} x^{m+2}=\sum_{s=1}^{m+2}C_s^{m+2}(1 - (-1)^s)(D_h)^m(x^{m+2-s}) \cdot \frac{h^{s-1}}{2} = 0\;\;\mbox{ for }\;\;x=0. 
$$         
Statement (\ref{third}) is established.

      Applying $(D_h)^j$ to the Taylor expansion
$$
f(x)-f(0)=\sum_{s=1}^{j+1} \frac{1}{s!}(\frac{d}{dx})^sf(0)\cdot x^s + \int_0^1 (\frac{d}{dt})^{j+2} f(t\cdot x)\frac{(1-t)^{j+1}} {(j+1)!} dt
$$
and taking into account (\ref{first}), (\ref{second})  yields
\begin{eqnarray}
\label{fourth}
(D_h)^jf(0) &=& (\frac{d}{dx})^jf(0) +\frac{1}{(j+1)!}(\frac{d}{dx})^{j+1}f(0)(D_h)^{j} x^{j+1}+ \nonumber\\
&&\\
&& \int_0^1 (D_h)^j((\frac{d}{dt})^{j+2} f(t\cdot x))\frac{(1-t)^{j+1}} {(j+1)!} dt \nonumber
\end{eqnarray}
On the other hand,
$$
D_h((\frac{d}{dt})^{j+2} f(t\cdot x)) = \frac{(\frac{d}{dt})^{j+2} f(t\cdot (x+h)) - (\frac{d}{dt})^{j+2} f(t\cdot (x-h))}{2 \cdot h} =(\frac{d}{dt})^{j+2}(t\cdot D_{th}f(tx))
$$
and so is
$$
(D_h)^j((\frac{d}{dt})^{j+2} f(t\cdot x)) = (\frac{d}{dt})^{j+2} (t^j (D_{th})^jf(tx))
$$
Therefore setting $x=0$ in (\ref{fourth}) and making use of (\ref{third}) we obtain
$$
(D_h)^jf(0) = (\frac{d}{dx})^jf(0) + \int_0^1 (\frac{d}{dt})^{j+2} (t^j\cdot (D_{th})^jf(0))\frac{(1-t)^{j+1}}{(j+1)!} dt
$$
That completes the calculation of (\ref{difference}) and the proof.

\end{proof}  

Given the required margin of error $\varepsilon >0$ Theorem \ref{real_analytic_LGA} provides an effective recipe for finding the global minimum of a real analytic function on the interval $[-1,\;1].$ 

\paragraph{\bf LGA for real analytic functions}

\begin{itemize}
\item[{\bf Step 1.}] Find a natural number $n$ so that  
$$
\mid \int_0^1 (\frac{d}{dx})^{n+1} f(t\cdot x)\cdot x^{n+1}\frac{(1-t)^{n}} {n!} dt \mid \le \frac{\varepsilon}{2}
$$
\item[{\bf Step 2}] Calculate a step size $h>0$ such that
$$
\mid \sum_{j=1}^n x^j\cdot \int_0^1 \frac{(1-t)^{j+1}}{j!\cdot (j+1)!}\cdot (\frac{d}{dt})^{j+2}(t^j\cdot D_{t\cdot h}^jf(0))dt \mid \le \frac{\varepsilon}{2}
$$
\item[{\bf Step 3}] Use LGA to solve optimization problem (\ref{optimization_problem_LGA}) with $h>0$ from {\bf Step 2}.
\end{itemize}

Let $x^{\star}$ be the $x$-argument of the global minimum calculated at {\bf Step 3}. Let $x_{orig}$ be the $x$-argument of the global minimum calculated for the original function $f(x)$ from (\ref{optimization_problem_analytic}). Then 
$$
\mid f(x_{orig}) - P_f^n(x^{\star})\mid \le \varepsilon
$$
and
$$
P_f^n(x_{orig}) \le P_f^n(x^{\star}) + 2\cdot \varepsilon .
$$


\section{Numerical experiments with univariate polynomials}
\label{numerical_experiments}

This section presents the results of numerical experiments conducted in order to compare the performance of LGA with Brute-Force Search (BFS), Zero Derivative Method (ZDM), modifications of Piyavskii-Shubert Method (PSM) discussed in \cite{Ellaia},  \cite{Sergeyev}, \cite{YSergeyev}. All numerical experiments follow the same scenario: 

\begin{itemize}
\item[$\bullet$] Repeat 500 times {\bf Step 1} and {\bf Step 2}. 
      \begin{itemize}
        \item[{\bf Step 1.}] Randomly generate a real polynomial $p(x)$ of degree $n$ with  roots uniformly distributed on $[-1,\;b]\times [-1,\;1],$ where $b$ is a real number between $-1$ and $1$ which remains fixed across all 500 trials.
       \item[{\bf Step 2.}] Use LGA and its competitor to solve the optimization problem
$$
p(x)\;\;\to\;\; \min_{[-1,\;1]}
$$ 
with precision $0.0001$ for $x$-argument of the global minimum on $[-1,\;1].$
        Record the processing time for  LGA and its competitor.
      \end{itemize}
\item[$\bullet$] After repeating 500 times ${\bf Step 1} \mbox{ and } {\bf Step 2}$ calculate $T_{\ell}$ and $T_c,$ the average processing time for LGA and its competitor respectively.
\item[$\bullet$] Update the file with the experimental records by adding a new line $(n,\;T_{\ell},\;T_c),$ where $n$ is the degree of the polynomial.
\end{itemize}

The final result is presented in the form of two curves (average time spent versus polynomial degree), one for LGA and the other for its competitor. 

Total time spent $T_c$ includes all necessary supplementary steps that are needed in order to successfully implement the tested algorithm. For example, ZDM total time covers calculation of critical points with Laguerre's method and the subsequent search for the minimum among critical values. PSM time includes calculation of the Lipschitz constant or its counterparts.

\subsection{BFS} 
  BFS attacks 
$$
f(x) \;\to\;\min_{[a,\;b]}
$$ 
by transforming it into 
$$
f(a + (b-a)\cdot\frac{j}{N})\;\to\;\;\min_{0\le j \le N}
$$
and then taking the smallest value in $\{f(a + (b-a)\cdot\frac{j}{N})\}_{j=0}^N$ and its respective $x$-argument as an estimate for the solution of the optimization problem.

LGA outperforms BFS for polynomials with roots uniformly distributed on $[-1,\;b]\times [-1,\;1]$ where $-1< \;b < 1.$ LGA considerably speeds up as the value of the parameter $b$ decreases (Fig. \ref{fig:2}, Fig. \ref{fig:3}). 

\input epsf
\setlength{\unitlength}{1cm}
\begin{figure}
\epsfxsize=11cm
\epsfysize=6cm
\epsfbox{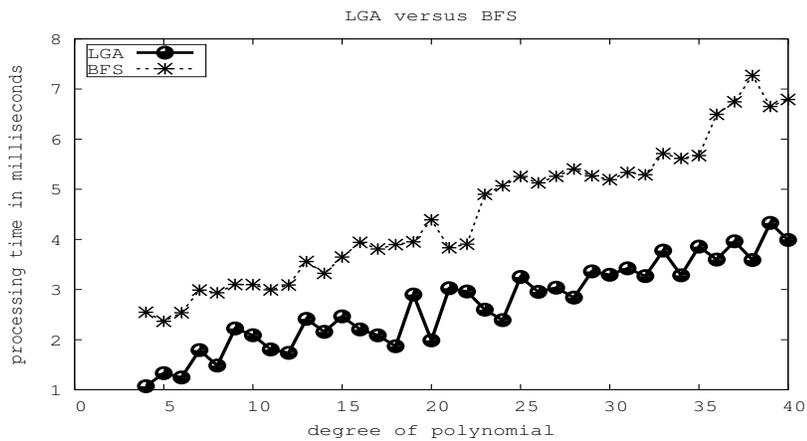}
\caption{LGA versus Brute-Force Search, where $b=0.$}
\label{fig:2}       
\end{figure}

\input epsf
\setlength{\unitlength}{1cm}
\begin{figure}
\epsfxsize=11cm
\epsfysize=6cm
\epsfbox{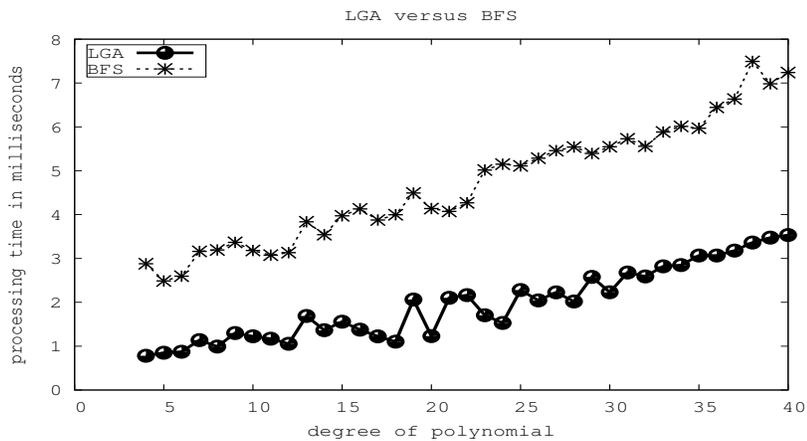}
\caption{LGA versus Brute-Force Search, where $b=-0.5.$}
\label{fig:3}       
\end{figure}

If $b=1$ then LGA works exactly so well as BFS (Fig. \ref{fig:4}, Fig. \ref{fig:5}). A generic application corresponds to the situation with $b<1.$ Therefore employing LGA instead of BFS  will improve the performance of your application.

\input epsf
\setlength{\unitlength}{1cm}
\begin{figure}
\epsfxsize=11cm
\epsfysize=6cm
\epsfbox{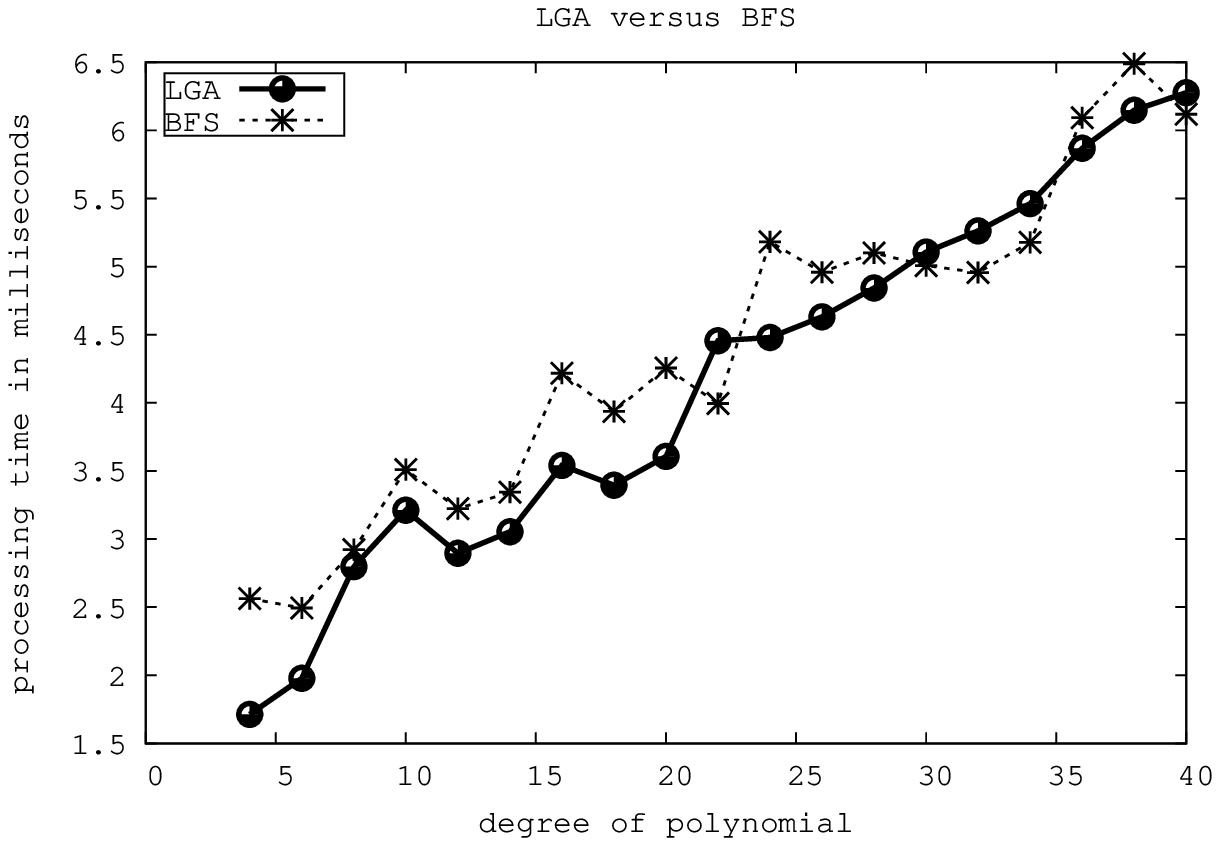}
\caption{LGA versus Brute-Force Search for polynomials of even degrees and where $b=1.$}
\label{fig:4}       
\end{figure}

\input epsf
\setlength{\unitlength}{1cm}
\begin{figure}
\epsfxsize=11cm
\epsfysize=6cm
\epsfbox{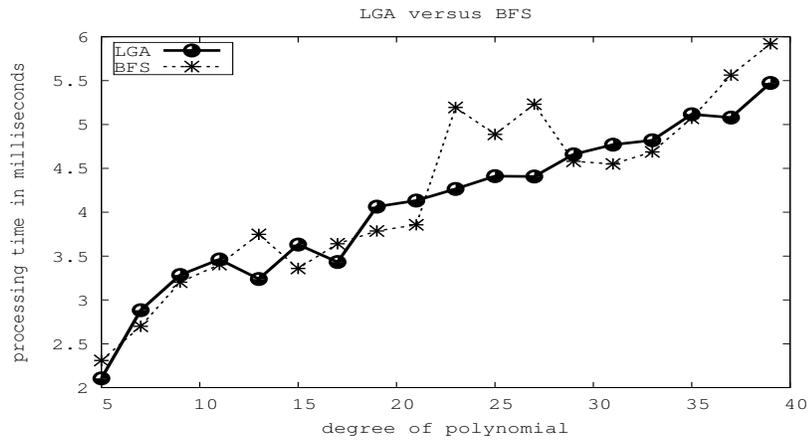}
\caption{LGA versus Brute-Force Search for polynomials with odd degrees and where $b=1.$}
\label{fig:5}       
\end{figure}

\subsection{ZDM}
ZDM finds the global minimum
$$
f(x) \;\to\;\min_{[a,\;b]}
$$
by calculating zeroes $\{ x_j \}_{j=1}^M $ of the derivative
$$
\frac{d}{dx}f(x) = 0\;\;\mbox{for}\;\;x \; \in\;  [a,\;b]
$$
and then finding the smallest value in $\{ f(x_j) \}_{j=1}^M .$ If it is $f(x_j)$ then ZDM returns $(x_j,\;f(x_j))$  as an estimate for the solution of the optimization problem.
In all numerical experiments reported in this paper the polynomial roots for ZDM were calculated with Laguerre's Method \cite{Adams}, \cite{Armijo}. The ZDM processing time  includes the invocation of Laguerre's Method. LGA notably faster than ZDM (Fig. \ref{fig:6}). 

\input epsf
\setlength{\unitlength}{1cm}
\begin{figure}
\epsfxsize=11cm
\epsfysize=6cm
\epsfbox{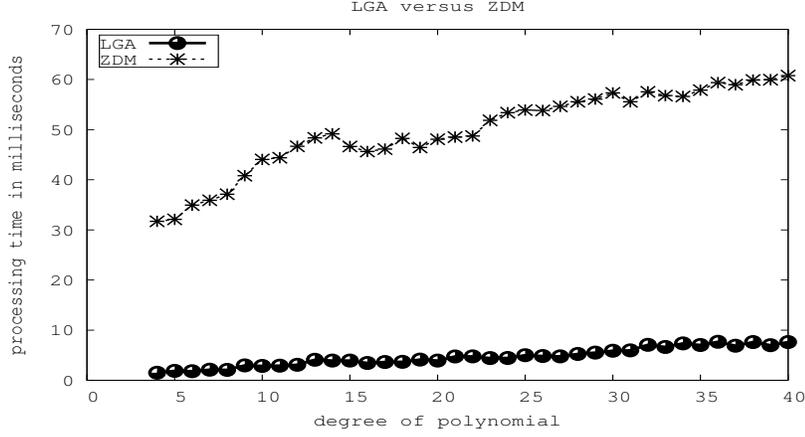}
\caption{LGA versus ZDM and $b=1.$}
\label{fig:6}       
\end{figure}

\subsection{PSM and its accelerations}
Piyavskii's type algorithms tackle  the optimization problem
$$
f(x) \;\;\to \;\; \min_{[a,\;b]}
$$
by constructing at each iteration either a piecewise linear $(f(x)$ is Lipschitz \cite{Piyavskii_67}, \cite{Piyavskii_72}, \cite{Sergeyev}, \cite{Shubert}) or a piecewise quadratic  $( \frac{d}{dx}f(x)$ is Lipschitz \cite{Ellaia}, \cite{Kvasov}, \cite{Sergeyev}, \cite{YSergeyev} ) auxiliary function $\Phi_n(x)$ such that
$$
\Phi_n(x) \le f(x) \;\;\forall \; x \in [a,\;b]
$$  
Then the original optimization problem is replaced with
$$
\Phi_n(x) \;\;\to \;\; \min_{[a,\;b]}
$$
Based on its solution the algorithm either terminates or proceeds to the next step with the new refined auxiliary function $\Phi_{n+1}(x).$

\subsubsection{Modifications of PSM with local tuning of piecewise linear auxiliary functions}

LGA is compared against PSM with tuning of the local Lipschitz constants  ( referred as LT) and its enhancement LT\_LI presented in \cite{Sergeyev}. In view of the numerical simulations from \cite{Sergeyev} LT and LT\_LI appear to be the fastest among Pyavskii's type algorithms (discussed in \cite{Sergeyev}) with piecewise linear auxiliary functions.
LGA is faster than LT LI (Fig. \ref{fig:7}). 

\input epsf
\setlength{\unitlength}{1cm}
\begin{figure}
\epsfxsize=11cm
\epsfysize=6cm
\epsfbox{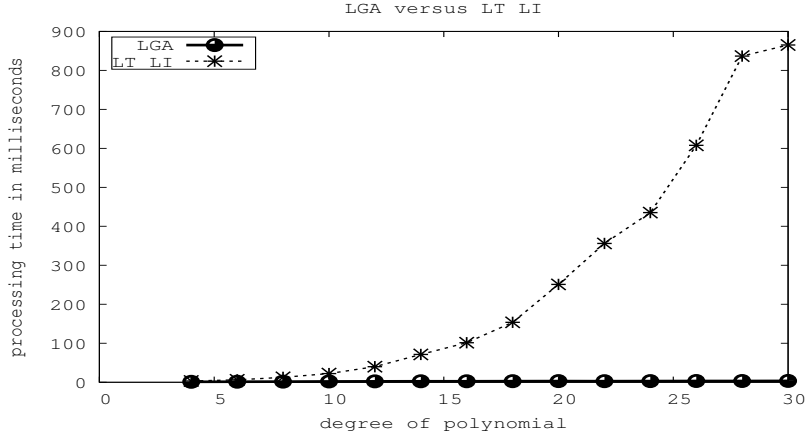}
\caption{LGA versus LT LI for polynomials with even degrees and $b=1.$}
\label{fig:7}       
\end{figure}

\subsubsection{Modifications of PSM with local tuning of piecewise quadratic auxiliary functions} 

Modifications of PSM with smooth piecewise quadratic auxiliary functions are discussed in \cite{Sergeyev}. Fig. \ref{fig:8} presents the  comparison results between LGA and PSM enriched by local tuning of the Lipschitz constant for $\frac{d}{dx}f(x)$ (referred as DLT). LGA outperforms DLT.

\input epsf
\setlength{\unitlength}{1cm}
\begin{figure}
\epsfxsize=11cm
\epsfysize=6cm
\epsfbox{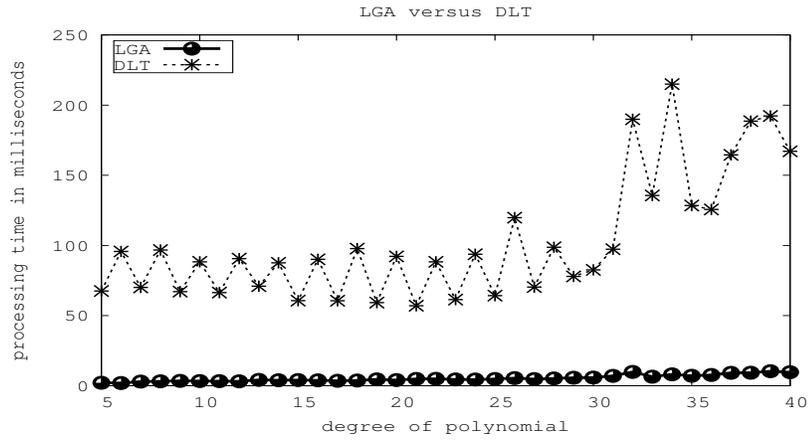}
\caption{LGA versus DLT, $b=1.$}
\label{fig:8}       
\end{figure}

  DLT with some local improvement technique \cite{Sergeyev} is addressed as DLT\_LI. Its comparison with LGA is presented by Fig.\ref{fig:9}. LGA is faster than DLT\_LI.
  
  \input epsf
\setlength{\unitlength}{1cm}
\begin{figure}
\epsfxsize=11cm
\epsfysize=6cm
\epsfbox{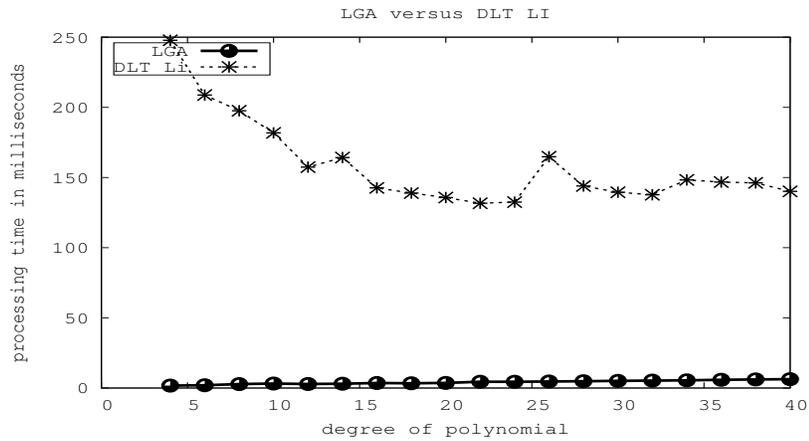}
\caption{LGA versus DLT LI for polynomials with even degrees, $b=1.$}
\label{fig:9}       
\end{figure}

\subsubsection{Modifications of PSM with piecewise quadratic auxiliary functions} 
The paper \cite{Ellaia} introduces the modification of PSM based on piecewise quadratic auxiliary functions that are not necessary smooth. The algorithm from \cite{Ellaia} is referred in Fig.\ref{fig:10} as EEK. LGA is faster than EEK. 

   \input epsf
\setlength{\unitlength}{1cm}
\begin{figure}
\epsfxsize=11cm
\epsfysize=6cm
\epsfbox{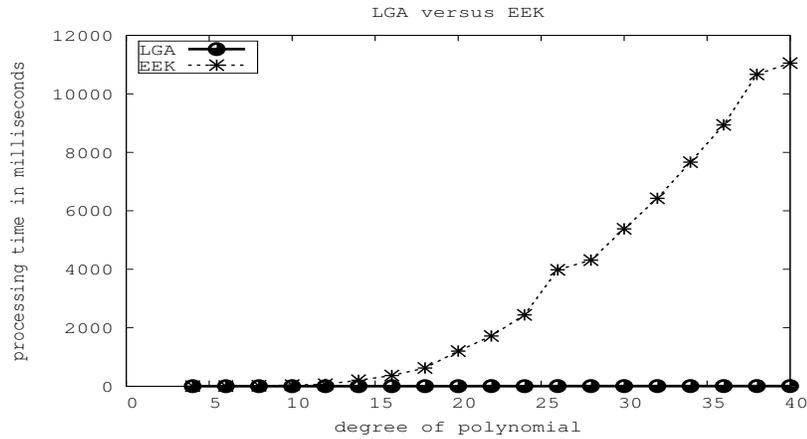}
\caption{LGA versus EEK for polynomials with even degrees, $b=1.$}
\label{fig:10}       
\end{figure}


\section{Acknowledgment}
\label{Acknowledgment}

The author is grateful to anonymous referees for comments and suggestions that helped to focus and improve the original manuscript. In particular, the comparison between LGA and Piyavskii-Shubert method was added upon a referee remark. The author is thankful to Blaiklen Marx for the help with preparing the manuscript for publication and testing the java implementations of algorithms in i-oblako.org framework.

\section{Appendix}
\label{appendix}
The program is looking for the global minimum of a polynomial on the interval $[a,\;b].$ The polynomial is represented as an array 
\begin{lstlisting}
double polynom = new double[degree+1];
\end{lstlisting}
A snapshot of the source code fragment essential for LGA is as follows.
\begin{lstlisting}
public double getMin(double[] polynom,
                     double a, double b,
                            double step){
double rt=a;
double[] pl = null;
double rt_prev=a;

if(a >= b)
   return b;

if(b-a <= step)
   return a;

while(polynom[polynom.length - 1] == 0 && polynom.length > 1){
     pl= new double[polynom.length - 1];
     for(int i =0; i < polynom.length - 1;i++)
           pl[i]=polynom[i];
    polynom=pl;
}
if(polynom.length == 1)
  return a;
if(polynom.length == 2){
  if(polynom[1]>= 0 )
        return a;
  else
        return b;
}

rt=a;
rt_prev=rt;
int Njumps = 0;
    do{
          rt_prev = GradientDescent(polynom,rt,b,step);
            if(rt_prev >= b)
                  return b;

           if(Njumps >= (polynom.length - 3))
                  return rt_prev;

          pl = Horner(polynom,rt_prev);
          rt = getMin(pl,rt_prev,b,step);
        if(rt - rt_prev<= step)
              return rt_prev;
        else{
        if(rt_prev == a)
          Njumps++;
        else
          Njumps=Njumps + 2;
         }

          if(HornerEval(pl,rt)>=0)
                 return rt_prev;

        rt_prev=rt;
      }
  while(rt < b-step);
}
\end{lstlisting} 

and gradient descent is implemented as follows.

\begin{lstlisting}
public double EvalDerivative(double[] polynom,double x){
          double ret = 0;

               for(int i = polynom.length -1; i>0; i--)
                     ret=ret*x + i*polynom[i];
               return ret;
}

public double GradientDescent(double[] polynom,
                              double a, double b,
                                     double step){
  
        double  rt=a;
       
          while(EvalDerivative(polynom,rt)<0 && rt < b )
                             rt=rt+step;
             return rt;
}

\end{lstlisting}

Java implementations of the algorithms discussed in the paper are available upon request.

%



\begin{thebibliography}{}
%
%
\bibitem{Adams}
D.A. Adams, "A Stopping Criterion for Polynomial Root Finding." Comm. ACM,10, 655-658, (1967).
\bibitem{Armijo}
L. Armijo, Minimization of functions having Lipschitz continuous first partial derivatives, Pacific J. Math., 16, 1 - 3 (1966)
\bibitem{Calvin} 
J. M. Calvin, An adaptive univariate global optimization algorithm and its convergence rate under the Wiener measure, Informatica, 22, 471–488 (2011).
\bibitem{Calvin_Zilinkas}
J. M. Calvin and A. \u{Z}ilinskas, One-dimensional global optimization for observations with noise, Comput. Math. Appl., 50, 157–169 (2005).
\bibitem{Ellaia}
R. Ellaia, M. Z. Es-Sadek, H. Kasbioui, Modified Piyavskii's Global One-Dimensional Optimization of a Differentiable Function, Applied Mathematics, 3, 1306-1320 (2012)
\bibitem{Higham}
N.J. Higham, Accuracy and Stability of Numerical Algorithms, SIAM, Philadelphia (2002).
\bibitem{Hamacher} K. Hamacher, On stochastic global optimization of one-dimensional functions, Phys. A, 354, 547–557 (2005).
\bibitem{Guler}
O. G\"{u}ler, Foundations of Optimizations, Springer, New York (2010)
\bibitem{Johnson} D. E. Johnson, Introduction to Filter Theory, Prentice Hall, New Jersey, (1976)
\bibitem{KantorovichAkilov}
L. Kantorovich and G. Akilov, Functional Analysis in Normed Spaces, Fizmatgiz, Moscow (1959), translated by D. Brown and A. Robertson, Pergamon Press, Oxford (1964)
\bibitem{Kalra}
D. Kalra and A. H. Barr, Guaranteed ray intersections with implicit surface, Comput. Graph., 23, 297–306 (1989)
\bibitem{Knuth}
D.E. Knuth, Art of Computer Programming, Vol. 2: Seminumerical Algorithms, 3rd ed., MA: Addison-Wesley, (1998)
\bibitem{Kvasov}
D.E. Kvasov, Y.D. Sergeyev, A Univariate Global Search Working With a Set of Lipschitz Constants for the First Derivative, Optim. Lett., 3, 303-318 (2009)
\bibitem{Lam} H. Y.-F. Lam, Analog and Digital Filters-Design and Realization, Prentice Hall, New Jersey, (1979)
\bibitem{Nesterov}
Yu. Nesterov, Introductory Lectures on Convex Optimization, Applied Optimization, 87, Kluwer Academic Publishers, Boston (2004)
\bibitem{Piyavskii_67}
S. Piyavskii, An algorithm for finding the absolute minimum of a function, Theory of Optimal Solutions, IK Akad. Nauk USSR, Kiev, 2, 13-24, (1967)
\bibitem{Piyavskii_72}
S. Piyavskii, An algorithm for finding the absolute extremum of a function, USSR Comput. Math.Math. Phys., 12, 57-67, (1972)
\bibitem{Shubert} 
B. Shubert, A sequential method seeking the global maximum of a function, SIAM J. Numer. Anal., 9, 379-388 (1972)
\bibitem{Sergeyev} 
D. Lera and Y. D. Sergeyev, Acceleration of Univariate Global Optimization Algorithms Working With Lipschitz Functions and Lipschitz First Derivatives, SIAM J. Optim., 23, 1, 508-529 (2013) 
\bibitem{YSergeyev} 
Y. D. Sergeyev, Global one-dimensional optimization using smooth auxiliary functions, Math. Programming, 81, 127-146 (1998)


\end{thebibliography}


\end{document}